\documentclass[11pt]{article}

\usepackage{epsfig,epsf,fancybox}
\usepackage{amsmath}
\usepackage{mathrsfs}
\usepackage{amssymb}
\usepackage{graphicx}
\usepackage{color}
\usepackage{multirow}

\newtheorem{theorem}{Theorem}[section]
\newtheorem{corollary}[theorem]{Corollary}
\newtheorem{lemma}[theorem]{Lemma}

\newtheorem{definition}{Definition}[section]

\newtheorem{remark}[theorem]{Remark}
\newtheorem{assumption}[theorem]{Assumption}
\numberwithin{equation}{section}

\newcommand{\st}{\textnormal{s.t.}}

\newcommand{\rank}{\textnormal{rank}}

\,
\,

\newcommand{\RR}{\mathbf R}
\newcommand{\argmin}{\mathop{\rm argmin}}
\newcommand{\LCal}{\mathcal{L}}

\newcommand{\half}{\frac{1}{2}}

\newcommand{\be}{\begin{equation}}
\newcommand{\ee}{\end{equation}}
\newcommand{\ba}{\begin{array}}
\newcommand{\ea}{\end{array}}
\newcommand{\bpm}{\begin{pmatrix}}
\newcommand{\epm}{\end{pmatrix}}

\newcommand{\XCal}{\mathcal{X}}

\newcommand{\etal}{{et al. }}

\begin{document}

\title{On the Global Linear Convergence of the ADMM \\ with Multi-Block Variables}

\author{Tianyi Lin\footnotemark[1]  \and Shiqian Ma\footnotemark[1] \and Shuzhong Zhang\footnotemark[2]}
\renewcommand{\thefootnote}{\fnsymbol{footnote}}
\footnotetext[1]{Department of Systems Engineering and Engineering Management, The Chinese University of Hong Kong, Shatin, New Territories, Hong Kong, China.}
\footnotetext[2]{Department of Industrial and Systems Engineering, University of Minnesota, Minneapolis, MN 55455, USA.}

\date{May 31, 2014, \quad Revised on May 21, 2015}

\maketitle

\begin{abstract}
The alternating direction method of multipliers (ADMM) has been widely used for solving structured convex optimization problems. In particular, the ADMM can solve convex programs that minimize the sum of $N$ convex functions with $N$-block variables linked by some linear constraints. While the convergence of the ADMM for $N=2$ was well established in the literature, it remained an open problem for a long time whether or not the ADMM for $N \ge 3$ is still convergent. Recently, it was shown in \cite{Chen-admm-failure-2013} that without further conditions the ADMM for $N\ge 3$ may actually fail to converge.
In this paper, we show that under some easily verifiable and reasonable  conditions the global linear convergence of the ADMM when $N\geq 3$ can still be assured, which is important since the ADMM is a popular method for solving large scale multi-block optimization models and is known to perform very well in practice even when $N\ge 3$. Our study aims to offer an explanation for this phenomenon.
\vspace{0.8cm}




\end{abstract}

\section{Introduction}

In this paper, we consider the global linear convergence of the standard alternating direction method of multipliers (ADMM) for solving convex minimization problems with $N$-block variables when $N\geq 3$. The problem under consideration can be formulated as
\be\label{prob:P-1}\ba{ll} \min & \tilde{f}_1(x_1) + \tilde{f}_2(x_2) + \cdots + \tilde{f}_N(x_N) \\
                         \st  & A_1 x_1  + A_2 x_2 + \cdots + A_N x_N = b, \\
                              & x_i \in \XCal_i, i=1,\ldots,N, \ea \ee
where $A_i \in \RR^{p\times n_i}$, $b\in\RR^p$, $\XCal_i\subset \RR^{n_i}$ are closed convex sets, and $\tilde{f}_i:\RR^{n_i}\rightarrow\RR^p$ are closed convex functions. Note that the convex constraint $x_i\in \XCal_i$ can be incorporated into the objective using an indicator function, i.e., \eqref{prob:P-1} can be rewritten as
\be\label{prob:P}\ba{ll} \min & {f}_1(x_1) + {f}_2(x_2) + \cdots + {f}_N(x_N) \\
                         \st  & A_1 x_1  + A_2 x_2 + \cdots + A_N x_N = b, \ea \ee
where $f_i(x_i):= \tilde{f}_i(x_i) + \mathbf{1}_i(x_i)$ and
\[\mathbf{1}_i(x_i):= \left\{ \ba{ll} 0 & \mbox{ if } x_i \in \XCal_i \\ +\infty & \mbox{ otherwise.} \ea \right.\]
We thus consider the equivalent reformulation \eqref{prob:P} throughout this paper for the ease of presentation.

For given $(x_2^k,\ldots,x_N^k;\lambda^k)$, a typical iteration of the ADMM for solving \eqref{prob:P} can be summarized as:
\be\label{admm-N}
\left\{\ba{lcl} x_1^{k+1} & := & \argmin_{x_1} \ \LCal_\gamma(x_1,x_2^k,\ldots,x_N^k;\lambda^k) \\
                x_2^{k+1} & := & \argmin_{x_2} \ \LCal_\gamma(x_1^{k+1},x_2,x_3^k,\ldots,x_N^k;\lambda^k) \\
                          & \vdots & \\
                x_N^{k+1} & := & \argmin_{x_N} \ \LCal_\gamma(x_1^{k+1},x_2^{k+1},\ldots,x_{N-1}^{k+1},x_N;\lambda^k) \\
                \lambda^{k+1} & := & \displaystyle\lambda^k - \gamma \left(\sum_{i=1}^N A_i x_i^{k+1} -b\right),           \ea\right. \ee
where
\[\LCal_\gamma(x_1,\ldots,x_N;\lambda) := \sum_{i=1}^N f_i(x_i) - \left\langle \lambda,\sum_{i=1}^NA_ix_i-b \right\rangle + \frac{\gamma}{2} \left\|\sum_{i=1}^NA_ix_i-b \right\|^2 \]
denotes the augmented Lagrangian function of \eqref{prob:P} with $\lambda$ being the Lagrange multiplier and $\gamma>0$ being a penalty parameter. It is noted that in each iteration, the ADMM updates the primal variables $x_1,\ldots,x_N$ in a Gauss-Seidel manner.

When $N=2$, the ADMM \eqref{admm-N} was shown to be equivalent to the Douglas-Rachford operator splitting method that dated back to 1950s for solving variational problems arising from PDEs \cite{Douglas-Rachford-56,Gabay-83}. The convergence of the ADMM \eqref{admm-N} when $N=2$ was thus established in the context of operator splitting methods \cite{Lions-Mercier-79,Eckstein-Bertsekas-1992}. Recently, ADMM has been revisited due to its success in solving structured convex optimization problems arising from sparse and low-rank optimization and related problems (we refer the readers to some recent survey papers for more details, see, e.g., \cite{Boyd-etal-ADM-survey-2011,Eckstein-tutorial-admm}). In \cite{Lions-Mercier-79}, Lions and Mercier showed that the Douglas-Rachford operator splitting method converges linearly under the assumption that some involved monotone operator is both coercive and Lipschitz. Eckstein and Bertsekas \cite{Eckstein-Bertsekas-LP-1990} showed the linear convergence of the ADMM \eqref{admm-N} with $N=2$ for solving linear programs, which depends on a bound on the largest iterate in the course of the algorithm. In a recent work by Deng and Yin \cite{Deng-Yin-2012}, a generalized ADMM was proposed in which some proximal terms were added to the two subproblems in \eqref{admm-N}, and it was shown that this generalized ADMM converges linearly under certain assumptions on the strong convexity of functions $f_1$ and $f_2$, and the rank of $A_1$ and $A_2$. For instance, one sufficient condition suggested in \cite{Deng-Yin-2012} that guarantees the linear convergence of the generalized ADMM is that $f_1$ and $f_2$ are both strongly convex, $\nabla f_2$ is Lipschitz continuous and $A_2$ is of full row rank. Han and Yuan \cite{Han-Yuan-linear-2013} and Boley \cite{Boley-2012} both studied the local linear convergence of ADMM \eqref{admm-N} when $N=2$ for solving quadratic programs. The result in \cite{Han-Yuan-linear-2013} was based on some error bound condition \cite{Luo-Tseng-error-bound-SIOPT-1992}, and the one given in \cite{Boley-2012} was obtained by first writing the ADMM as a matrix recurrence and then performing a spectral analysis on the recurrence. Moreover, it was shown that the ADMM \eqref{admm-N} when $N=2$ converges sublinearly under the simple convexity assumption both in ergodic and non-ergodic sense \cite{He-Yuan-rate-ADM-2012,Monteiro-Svaiter-2010a,He-Yuan-nonergodic-2012}. It should be noted that all the convergence results on the ADMM \eqref{admm-N} discussed above are for the case $N=2$.

While the convergence properties of the ADMM when $N=2$ have been well studied, its convergence when $N\geq 3$ has remained unclear for a very long time. The following includes some recent progresses on this direction. In a recent work by Chen \etal \cite{Chen-admm-failure-2013}, a counter-example was given which shows that without further conditions the ADMM for $N\ge 3$ may actually fail to converge. Existing works that study sufficient conditions ensuring the convergence of ADMM when $N\geq 3$ are briefly summarized as follows. Han and Yuan \cite{Han-Yuan-note-2012} proved the global convergence of ADMM \eqref{admm-N} under the condition that $f_1, \ldots, f_N$ are all strongly convex and $\gamma$ is restricted to certain region.
Recently, Chen, Shen and You \cite{Chen-Shen-You-convergence-2013} and Lin, Ma and Zhang \cite{Lin-Ma-Zhang-admm-sublinear-2014} relaxed this condition to require only $N-1$ functions to be strongly convex and $\gamma$ is restricted to certain region. Lin, Ma and Zhang \cite{Lin-Ma-Zhang-admm-sublinear-2014} also showed that the ADMM \eqref{admm-N} possesses sublinear convergence rate in both ergodic and non-ergodic sense under these conditions. Closely related to \cite{Chen-Shen-You-convergence-2013, Lin-Ma-Zhang-admm-sublinear-2014}, Cai, Han and Yuan \cite{Cai-Han-Yuan-direct-2014} and Li, Sun and Toh \cite{Li-Sun-Toh-convergent-2015} proved the convergence of the 3-block (i.e., $N=3$) ADMM \eqref{admm-N} under the assumption that only one of the functions $f_1$, $f_2$ and $f_3$ is strongly convex, and $\gamma$ is restricted to be smaller than a certain bound. Davis and Yin \cite{Davis-Yin-2015} studied a variant of the 3-block ADMM (see Algorithm 8 in \cite{Davis-Yin-2015}) which requires that $f_1$ is strongly convex and $\gamma$ is smaller than a certain bound to guarantee the convergence. More recently, Lin, Ma and Zhang \cite{Lin-Ma-Zhang-2015} further proposed several alternative approaches to ensure the sublinear convergence rate of \eqref{admm-N} without requiring any function to be strongly convex. Furthermore, Lin, Ma and Zhang \cite{Lin-Ma-Zhang-2015-free-gamma} proved that the 3-block ADMM is globally convergent for any $\gamma>0$ when it is applied to solve a class of regularized least squares decomposition problems.

Along another line, there are some works that study variants of multi-block ADMM \eqref{admm-N}. Hong and Luo \cite{Luo-ADMM-2012} proposed to adopt a small step size when updating the Lagrange multiplier $\lambda^k$ in \eqref{admm-N}, i.e., they suggested that the update for $\lambda^k$, i.e.,
\be\label{update-lambda} \lambda^{k+1} := \lambda^k -\gamma \left(\sum_{i=1}^N A_i x_i^{k+1} -b\right), \ee
be changed to
\be\label{update-lambda-small-step} \lambda^{k+1} := \lambda^k - \alpha \gamma \left(\sum_{i=1}^N A_i x_i^{k+1} -b\right), \ee
where $\alpha>0$ is a small step size. It was shown in \cite{Luo-ADMM-2012} that this variant of ADMM converges linearly under the assumption that certain error bound condition holds and $\alpha$ is bounded by some constant that is related to the error bound condition. He, Tao and Yuan studied multi-block ADMM with additional correction steps \cite{He-Tao-Yuan-MOR-2013,He-Tao-Yuan-2012}. Parallel and Jacobian type variants of ADMM were studied in \cite{Goldfarb-Ma-Ksplit-SIOPT,Deng-admm-2014,He-Hou-Yuan-Jacob-2013}. A randomly permuted ADMM were recently proposed in \cite{Sun-Luo-Ye-random-admm-2015}.

{\bf Our contribution.} In this paper, we show the global linear convergence of ADMM \eqref{admm-N} when $N\geq 3$.
It should be noted that the linear convergence results in \cite{Lions-Mercier-79,Deng-Yin-2012,Han-Yuan-linear-2013,Boley-2012} are for the case $N=2$, while ours consider the case when $N\geq 3$. Moreover, compared with the {\em local} linear convergence results in \cite{Han-Yuan-linear-2013} and \cite{Boley-2012} for $N=2$, we prove the {\em global} linear convergence for $N\geq 3$. Furthermore, our result is for the {\em original standard multi-block ADMM} \eqref{admm-N}, while the one presented in \cite{Luo-ADMM-2012} is a variant of \eqref{admm-N} which replaces \eqref{update-lambda} with \eqref{update-lambda-small-step}.
To the best of our knowledge, our results in this paper are the first global linear convergence results for the {\em original standard multi-block ADMM} \eqref{admm-N} when $N\geq 3$.

The rest of this paper is organized as follows. In Section \ref{sec:pre}, we provide some preliminaries and prove three technical lemmas for the subsequent analysis. In Section \ref{sec:lin-con}, we prove the global linear convergence of ADMM \eqref{admm-N} under three different scenarios. Finally, we conclude the paper in Section \ref{sec:conclusion}.

\section{Preliminaries and Technical Lemmas}\label{sec:pre}

Although lacking convergence guarantee under the standard convexity assumption, it has been observed by many researchers that the Gauss-Seidel multi-block ADMM \eqref{admm-N} often outperforms all its modified versions in practice (see \cite{Sun-Toh-Yang-admm4-2014} and the numerical experiments in \cite{Wang-etal-2013-multi-2}). In particular, the numerical experiments on basis pursuit problem conducted in \cite{Wang-etal-2013-multi-2} show that the Gauss-Seidel multi-block ADMM \eqref{admm-N} is much more efficient than the ADM algorithm in \cite{Yang-Zhang-2009} and the variable splitting ADMMs studied in \cite{Bertsekas-Tsitsiklis-1989,Boyd-etal-ADM-survey-2011,Wang-etal-2013-multi-2}. In this section, we further provide some comparison results on Gauss-Seidel multi-block ADMM \eqref{admm-N} with Jacobian multi-block ADMM to motivate the necessity of studying \eqref{admm-N}. Note that Jacobian multi-block ADMM updates the primal variables in \eqref{admm-N} in a Jacobian manner.
For given $(x_2^k,\ldots,x_N^k;\lambda^k)$, a typical iteration of the Jacobian multi-block ADMM for solving \eqref{prob:P} can be summarized as:
\be\label{admm-N-Jacobian}
\left\{\ba{lcl} x_1^{k+1} & := & \argmin_{x_1} \ \LCal_\gamma(x_1,x_2^k,\ldots,x_N^k;\lambda^k) \\
                x_2^{k+1} & := & \argmin_{x_2} \ \LCal_\gamma(x_1^k,x_2,x_3^k,\ldots,x_N^k;\lambda^k) \\
                          & \vdots & \\
                x_N^{k+1} & := & \argmin_{x_N} \ \LCal_\gamma(x_1^k,x_2^k,\ldots,x_{N-1}^k,x_N;\lambda^k) \\
                \lambda^{k+1} & := & \displaystyle\lambda^k - \gamma \left(\sum_{i=1}^N A_i x_i^{k+1} -b\right).           \ea\right. \ee
Intuitively, the performance of Jacobian ADMM \eqref{admm-N-Jacobian} should be worse than the Gauss-Seidel version \eqref{admm-N}, because the latter one always uses the latest information of the primal variables in the updates. Now that \eqref{admm-N} is known to be  divergent (see \cite{Chen-admm-failure-2013}), it is therefore not surprising that the Jacobian ADMM \eqref{admm-N-Jacobian} may also be divergent. In fact, an example was given in \cite{He-Hou-Yuan-Jacob-2013} showing that the Jacobian ADMM \eqref{admm-N-Jacobian} is divergent even when $N=2$. A variant of the Jacobian ADMM \eqref{admm-N-Jacobian}, called proximal Jacobian ADMM, was recently proposed by Deng \etal \cite{Deng-admm-2014}. The proximal Jacobian ADMM iterates as follows.
\be\label{admm-N-Jacobian-Prox}
\left\{\ba{lcl} x_1^{k+1} & := & \argmin_{x_1} \ \LCal_\gamma(x_1,x_2^k,\ldots,x_N^k;\lambda^k) + \half\|x_1-x_1^k\|_{P_1}^2 \\
                x_2^{k+1} & := & \argmin_{x_2} \ \LCal_\gamma(x_1^k,x_2,x_3^k,\ldots,x_N^k;\lambda^k) + \half\|x_2-x_2^k\|_{P_2}^2\\
                          & \vdots & \\
                x_N^{k+1} & := & \argmin_{x_N} \ \LCal_\gamma(x_1^k,x_2^k,\ldots,x_{N-1}^k,x_N;\lambda^k) + \half\|x_N-x_N^k\|_{P_N}^2\\
                \lambda^{k+1} & := & \displaystyle\lambda^k - \alpha\gamma \left(\sum_{i=1}^N A_i x_i^{k+1} -b\right),           \ea\right. \ee
                where $P_i, i=1,\ldots,N$ are positive semidefinite matrices, and $\alpha>0$ is a dual step size. The $o(1/k)$ iteration complexity of proximal Jacobian ADMM \eqref{admm-N-Jacobian-Prox} was obtained in \cite{Deng-admm-2014} under some conditions on $\gamma$, $\alpha$, and $P_i,i=1,\ldots,N$. In the following we conduct some comparison on the Gauss-Seidel multi-block ADMM \eqref{admm-N}, the proximal Jacobian ADMM (Prox-JADMM) \eqref{admm-N-Jacobian-Prox}, the Jacobian ADMM with correction step (Corr-JADMM) \cite{He-Hou-Yuan-Jacob-2013}, and the variable-splitting ADMM (VSADMM) \cite{Wang-etal-2013-multi-2}.
               We compare the performance of these four algorithms for solving the basis pursuit problem
\be\label{bp}\min \ \|x\|_1, \ \st, Ax=b,\ee
where $A\in \RR^{p\times N}$ and $b\in \RR^p$. Note that \eqref{bp} is in the form of \eqref{prob:P} if we define $f_i(x_i) = |x_i|$, $i=1,\ldots,N$. In our experiments, we used the same experimental settings as in \cite{Deng-admm-2014}: the entries of matrix $A$ were randomly generated following standard Gaussian distribution; the true solution $x^*$ was also generated randomly following standard Gaussian distribution with sparsity level $s$ (number of nonzeros entries of $x^*$); $b$ was then generated by $b=Ax^*+\mathbf{n}$, where   $\mathbf{n}\approx\mathcal{N}(0,\sigma^2\mathbf{I})$ is Gaussian noise with standard deviation $\sigma$.
We chose $\gamma=10/\|b\|_1$. The Matlab codes of Prox-JADMM, Corr-JADMM and VSADMM were downloaded from the author's website of \cite{Deng-admm-2014}\footnote{http://www.math.ucla.edu/$\sim$zhimin.peng/parallel\_ADMM.html}.
We terminated all these four algorithms after running 200 iterations each. As in \cite{Deng-admm-2014}, we set $p=300$, $N=1000$ and $s=60$, and randomly generated 10 instances for both noise free case ($\sigma = 0$) and a noisy case ($\sigma=10^{-3}$). We report the number of iterations needed by these four algorithms to reduce the relative error $\|x-x^*\|/\|x^*\|$ below the given tolerance $\epsilon$ in Table \ref{table:GS-Jacobian}. Note that since the maximum number of iteration is 200, we put 200 in Table \ref{table:GS-Jacobian} for the cases where the algorithm could not reduce the relative error below $\epsilon$ in 200 iterations. From Table \ref{table:GS-Jacobian} we can see that the number of iterations needed by the Gauss-Seidel multi-block ADMM \eqref{admm-N} to reach the targeted accuracy is always smallest among the four compared algorithms (except one instance in the noise free case where none of the four algorithms can achieve the given tolerance in 200 iterations). We also plotted the geometric mean of the relative errors produced by the four algorithms over the 10 random instances in Figure \ref{fig:compare-ADMM}, from which we can see more clearly the advantage of the Gauss-Seidel multi-block ADMM \eqref{admm-N}. We need to remark here that one advantage of the Jacobian type ADMM variants is that the subproblems can be computed in parallel. As a result, if a parallel computing environment is available, then the Jacobian type variants of ADMM: Prox-JADMM and Corr-JADMM, can be faster than the Gauss-Seidel multi-block ADMM \eqref{admm-N}.

\begin{table}\caption{Number of iterations needed by Gauss-Seidel multi-block ADMM (GS) \eqref{admm-N}, proximal Jacobian ADMM (Prox-J) \eqref{admm-N-Jacobian-Prox}, Jacobian ADMM with correction step (Corr-J) \cite{He-Hou-Yuan-Jacob-2013} and variable-splitting ADMM (VS) \cite{Wang-etal-2013-multi-2} to achieve the targeted tolerance $\epsilon$ for basis pursuit problem \eqref{bp}.}
\begin{center}
\begin{small}
\begin{tabular}{|c|cccc|cccc|}\hline
Noise Level & GS & Prox-J & Corr-J & VS & GS & Prox-J & Corr-J & VS \\\hline
 {$\sigma = 0$}  & \multicolumn{4}{|c|}{ $\epsilon = 10^{-4}$} & \multicolumn{4}{|c|}{ $\epsilon = 10^{-6}$} \\\hline
&    69  &  85 &  116 &  200  &  81 &  118  & 145  & 200  \\
&    69 &  101 &  137 &  200  &  80 &  135  & 168  & 200  \\
&    46  &  71 &   83 &  200  &  57 &  109  & 112  & 200  \\
&    64 &  107 &  125 &  200  &  73 &  137  & 154  & 200  \\
&   200 &  200 &  200 &  200  & 200 &  200  & 200  & 200  \\
&   135 &  193 &  200 &  200  & 147 &  200  & 200  & 200  \\
&    76 &  114 &  146 &  200  &  86 &  150  & 169  & 200  \\
&    54 &   85 &  104 &  200  &  65 &  121  & 132  & 200  \\
&    97 &  125 &  165 &  200  & 109 &  173  & 195  & 200  \\
&    34 &   71 &   61 &  200  & 200 &  200  & 200  & 200  \\\hline
 {$\sigma = 10^{-3}$}  & \multicolumn{4}{|c|}{ $\epsilon = 10^{-2}$} & \multicolumn{4}{|c|}{ $\epsilon = 10^{-3}$} \\\hline
& 17  &  33 &   35  & 200 &   52  &  68  &  93  & 200  \\
                                                  & 18  &  33 &   30  & 200 &   35  &  45  &  56  & 200  \\
                                                  & 18  &  34 &   35  & 200 &   22  &  47  &  45  & 200  \\
                                                  & 19  &  45 &   45  & 200 &  125  & 160  & 200  & 200  \\
                                                  & 15  &  33 &   31  & 200 &   19  &  53  &  44  & 200  \\
                                                  & 11  &  34 &   24  & 200 &   43  &  75  &  75  & 200  \\
                                                  & 15  &  34 &   28  & 200 &   75  &  87  & 132  & 200  \\
                                                  & 14  &  32 &   25  & 200 &   43  &  68  &  78  & 200  \\
                                                  & 12  &  35 &   23  & 200 &   32  &  55  &  52  & 200  \\
                                                  & 31  &  45 &   63  & 200 &   71  &  87  & 131  & 200  \\\hline
\end{tabular}\label{table:GS-Jacobian}
\end{small}
\end{center}
\end{table}

\begin{figure}[htbp]
\centering{
\includegraphics[scale = 0.5]{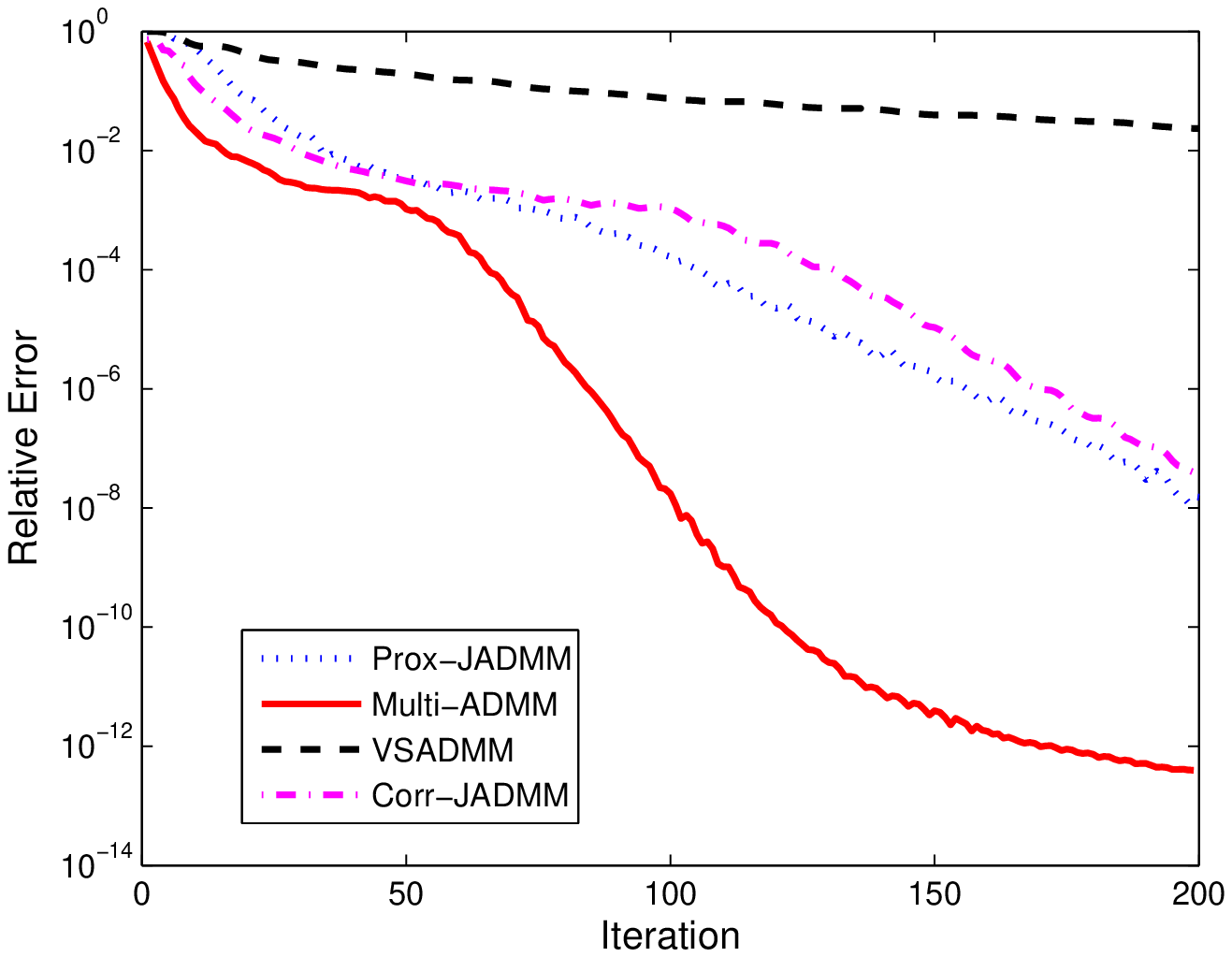}
\includegraphics[scale = 0.5]{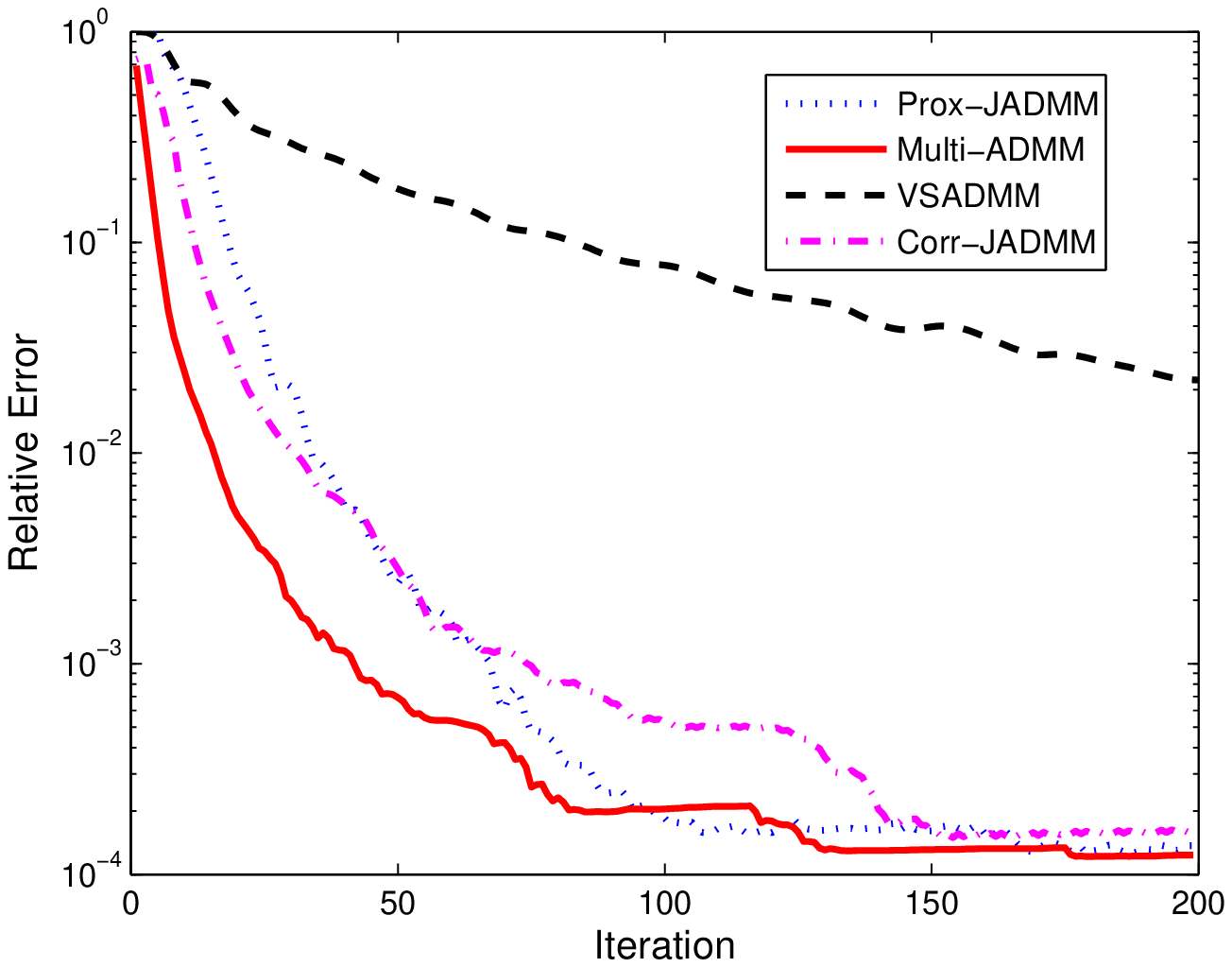}
}\caption{Comparison of the relative errors produced by the four algorithms for solving basis pursuit problem \eqref{bp}. Left: noise free. Right: noise level $\sigma=10^{-3}$. }
\label{fig:compare-ADMM}
\end{figure}

We now turn our attention to analyzing the convergence rate of the Gauss-Seidel multi-block ADMM \eqref{admm-N}.
We use $\Omega^* \subset \XCal_1\times \XCal_2 \times \ldots \times \XCal_N \times \RR^p$ to denote the set of primal-dual optimal solutions of \eqref{prob:P}.  Note that according to the first-order optimality conditions for \eqref{prob:P}, solving \eqref{prob:P} is equivalent to finding $$(x_1^*,\cdots,x_N^*,\lambda^*)\in\Omega^*$$
such that the followings hold:
\begin{align}
& A_{i}^{\top}\lambda^{*}\in\partial f_{i}(x_{i}^{*}), i=1,2,\cdots,N, \label{kkt-1}\\
& \sum\limits_{i=1}^{N}A_{i}x_{i}^{*}-b=0. \label{kkt-2}
\end{align}

We thus make the following assumption throughout this paper.
\begin{assumption}\label{exist} The optimal set $\Omega^*$ for problem \eqref{prob:P} is non-empty.\end{assumption}


In our analysis, the following well-known identity is used frequently:
\begin{eqnarray}\label{identity-4}
 (w_1-w_2)^{\top}(w_3-w_4) = \frac{1}{2}\left(\|w_1-w_4\|^{2}-\|w_1-w_3\|^{2}\right)+\frac{1}{2}\left(\|w_2-w_3\|^{2}-\|w_2-w_4\|^{2}\right).
\end{eqnarray}

{\bf Notations.} We use $g_i$ to denote a subgradient of $f_i$; $\lambda_{\max}(B)$ and $\lambda_{\min}(B)$ denote respectively the largest and smallest eigenvalues of a real symmetric matrix $B$; $\|x\|$ denotes the Euclidean norm of $x$. We use $\sigma_i>0$ to denote the convexity parameter of $f_i$, i.e., the following inequalities hold for $i=1,\ldots,N$:
\begin{eqnarray}\label{f-strong}
(x-y)^\top (g_i(x)-g_i(y)) \geq \sigma_i \|x-y\|^2, \quad \forall x,y\in\XCal_i,
\end{eqnarray}
where $g_i(x)\in \partial f_i(x)$ is a subgradient of $f_i$. Note that
$f_i$ is strongly convex if and only if $\sigma_i>0$, and if $f_i$ is convex but not strongly convex, then $\sigma_i=0$.

In this paper, we consider three scenarios that lead to global linear convergence of ADMM \eqref{admm-N}. The conditions of the three scenarios are listed in Table \ref{tab:3-scenarios}.

\begin{table}[htdp]
\begin{tabular}{c|c|c|c|c}\hline
scenario & strongly convex & Lipschitz continuous & full row rank & full column rank \\ \hline\hline
1 & $f_{2},\cdots,f_{N}$ & $\nabla f_{N}$ & $A_N$ & --- \\ \hline
2 & $f_{1},\cdots,f_{N}$ & $\nabla f_{1},\cdots,\nabla f_{N}$ & --- & --- \\ \hline
3 & $f_{2},\cdots,f_{N}$ & $\nabla f_{1},\cdots,\nabla f_{N}$ & --- & $A_{1}$ \\ \hline
\end{tabular}
\caption{Three scenarios leading to global linear convergence}\label{tab:3-scenarios}
\end{table}

We remark here that when $N=2$, the three scenarios listed in Table \ref{tab:3-scenarios} actually reduce to the same conditions considered by Deng and Yin as scenarios 1, 4 and 3, respectively in \cite{Deng-Yin-2012}. {We also remark here that since we incorporated the indicator functions into the objective function in \eqref{prob:P}, scenario 1 actually requires that there is no constraint $x_N\in\XCal_N$; scenarios 2 and 3 require that there is no constraint  $x_i\in\XCal_i$, $i=1,\ldots,N$. }

The first-order optimality conditions for the $N$ subproblems in \eqref{admm-N} are given by
\begin{eqnarray}\label{opt-x}
A_{i}^{\top}\lambda^{k}-\gamma A_{i}^{\top}\left(\sum\limits_{j=1}^{i}A_{j}x_{j}^{k+1}+\sum\limits_{j=i+1}^{N}A_{j}x_{j}^{k}-b\right)\in\partial f_{i}(x_{i}^{k+1}),\quad i=1,2,\cdots, N,
\end{eqnarray}
where we have adopted the convention $\sum_{j=N+1}^N a_j = 0$.
By combining with the updating formula for $\lambda^k$ \eqref{update-lambda}, \eqref{opt-x} can be rewritten as
\begin{eqnarray}\label{opt-x-lambda}
& A_{i}^{\top}\lambda^{k+1}-\gamma A_{i}^{\top}\left[\sum\limits_{j=i+1}^{N}A_{j}(x_{j}^{k}-x_{j}^{k+1})\right]\in\partial f_{i}(x_{i}^{k+1}),\quad i=1,2,\cdots, N.
\end{eqnarray}

Before we present the linear convergence of ADMM \eqref{admm-N}, we prove the following three technical lemmas that will be used in subsequent analysis.

\begin{lemma}\label{lemma1}
Let $(x_1^*,\ldots,x_N^*,\lambda^*)\in\Omega^*$. The sequence $\{x_1^k, x_2^k, \ldots, x_N^k, \lambda^k\}$ generated via ADMM \eqref{admm-N} satisfies,
\begin{eqnarray}\label{lemma-inequality}
\begin{aligned}
& \left(\frac{\gamma}{2}\sum\limits_{i=1}^{N-1}\left\| \sum\limits_{j=i+1}^{N}A_{j}(x_{j}^{*}-x_{j}^{k})\right\|^{2}+\frac{1}{2\gamma}\|\lambda^*-\lambda^{k}\|^{2}\right)  -\left(\frac{\gamma}{2}\sum\limits_{i=1}^{N-1}\left\| \sum\limits_{j=i+1}^{N}A_{j}(x_{j}^{*}-x_{j}^{k+1})\right\|^{2} +\frac{1}{2\gamma}\|\lambda^*-\lambda^{k+1}\|^{2}\right)\\
\geq & \sum\limits_{i=1}^{N-1}\left[\left(\sigma_i - \frac{\gamma(2N-i)(i-1)}{4}\lambda_{\max}(A_{i}^{\top}A_{i})\right)\| x_{i}^{k+1}-x_{i}^{*}\|^{2}\right]\\
& +\left(\sigma_N-\frac{\gamma(N+1)(N-2)}{4}\lambda_{\max}(A_{N}^{\top}A_{N})\right)\| x_{N}^{k+1}-x_{N}^{*}\|^{2} + \frac{\gamma}{2}\left\|A_1x_1^{k+1}+\sum_{j=2}^N A_j x_j^k-b\right\|^2.
\end{aligned}
\end{eqnarray}
\end{lemma}

\begin{proof}
Combining \eqref{opt-x-lambda}, \eqref{kkt-1} and \eqref{f-strong} yields,
\be\label{fi-opt}
(x_{i}^{k+1}-x_{i}^*)^{\top}A_{i}^\top\left(\lambda^{k+1}-\lambda^*-\gamma\sum\limits_{j=i+1}^{N}A_{j}(x_{j}^{k}-x_{j}^{k+1})\right) \geq \sigma_i \| x_{i}^{k+1}-x_{i}^*\|^2, \quad i=1,\ldots,N.
\ee

From \eqref{update-lambda} and \eqref{kkt-2}, it is easy to obtain
\be\label{lambda-opt}
\sum\limits_{i=1}^{N}A_{i}(x_{i}^{k+1}-x_{i}^*)=\frac{1}{\gamma}(\lambda^k-\lambda^{k+1}).
\ee

Summing \eqref{fi-opt} over $i=1,\cdots,N$ and using \eqref{lambda-opt}, we can get
\be\label{sum-1-N}
\frac{1}{\gamma}(\lambda^k-\lambda^{k+1})^{\top}(\lambda^{k+1}-\lambda^*)+\gamma\sum\limits_{i=1}^{N-1}(x_{i}^*-x_{i}^{k+1})^{\top}A_{i}^\top\left[ \sum\limits_{j=i+1}^{N}A_{j}(x_{j}^{k}-x_{j}^{k+1})\right] \geq \sum\limits_{i=1}^{N}\sigma_i \|x_{i}^{k+1}-x_{i}^*\|^2.
\ee

By adopting the convention $\sum_{i=1}^0 a_i = 0$, we have that
\begin{eqnarray}
\nonumber & & \sum\limits_{i=1}^{N-1}(x_{i}^*-x_{i}^{k+1})^{\top}A_{i}^\top\left[ \sum\limits_{j=i+1}^{N}A_{j}(x_{j}^{k}-x_{j}^{k+1})\right] \\
\nonumber & = & \sum\limits_{i=1}^{N-1}\left[\left(\sum\limits_{j=1}^{i}A_{j}x_{j}^*-b\right)-\left(\sum\limits_{j=1}^{i-1}A_{j}x_{j}^*+A_{i}x_{i}^{k+1}-b\right)\right]^{\top}
\left[\left(-\sum\limits_{j=i+1}^{N}A_{j}x_{j}^{k+1}\right)-\left(-\sum\limits_{j=i+1}^{N}A_{j}x_{j}^{k}\right)\right] \\
\nonumber & = & \sum\limits_{i=1}^{N-1}\left[\frac{1}{2}\left(\left\| \sum\limits_{j=1}^{i}A_{j}x_{j}^*+\sum\limits_{j=i+1}^{N}A_{j}x_{j}^{k}-b\right\|^{2}-\left\| \sum\limits_{j=1}^{i}A_{j}x_{j}^*+\sum\limits_{j=i+1}^{N}A_{j}x_{j}^{k+1}-b\right\|^{2}\right)\right.\\
\nonumber &  & \left.+\frac{1}{2}\left(\left\| \sum\limits_{j=1}^{i-1}A_{j}x_{j}^*+\sum\limits_{j=i}^{N}A_{j}x_{j}^{k+1}-b\right\|^{2} -\left\| \sum\limits_{j=1}^{i-1}A_{j}x_{j}^*+A_{i}x_{i}^{k+1}+\sum\limits_{j=i+1}^{N}A_{j}x_{j}^{k}-b\right\|^{2}\right)\right]\\
\nonumber & \leq & \frac{1}{2} \sum\limits_{i=1}^{N-1}\left(\left\| \sum\limits_{j=1}^{i}A_{j}x_{j}^*+\sum\limits_{j=i+1}^{N}A_{j}x_{j}^{k}-b\right\|^{2}-\left\| \sum\limits_{j=1}^{i}A_{j}x_{j}^*+\sum\limits_{j=i+1}^{N}A_{j}x_{j}^{k+1}-b\right\|^{2} \right) \\ \nonumber & & + \frac{1}{2}\sum_{i=1}^{N-1}\left\|\sum\limits_{j=1}^{i-1}A_{j}x_{j}^*+\sum\limits_{j=i}^{N}A_{j}x_{j}^{k+1}-b\right\|^{2} - \frac{1}{2}\left\|A_1x_1^{k+1}+\sum_{j=2}^N A_j x_j^k-b\right\|^2\\
\nonumber & = & \frac{1}{2}\sum\limits_{i=1}^{N-1}\left(\left\| \sum\limits_{j=1}^{i}A_{j}x_{j}^*+\sum\limits_{j=i+1}^{N}A_{j}x_{j}^{k}-b\right\|^{2}-\left\| \sum\limits_{j=1}^{i}A_{j}x_{j}^*+\sum\limits_{j=i+1}^{N}A_{j}x_{j}^{k+1}-b\right\|^{2}\right)\\
&  &  + \frac{1}{2\gamma^{2}}\|\lambda^{k+1}-\lambda^{k}\|^{2} + \frac{1}{2}\sum\limits_{i=2}^{N-1}\left\| \sum\limits_{j=1}^{i-1}A_{j}x_{j}^*+\sum\limits_{j=i}^{N}A_{j}x_{j}^{k+1}-b\right\|^{2}- \frac{1}{2}\left\|A_1x_1^{k+1}+\sum_{j=2}^N A_j x_j^k-b\right\|^2, \label{long-eq}
\end{eqnarray}
where in the second equality we have used the identity \eqref{identity-4}, and the last equality follows from \eqref{update-lambda}.

By combining \eqref{sum-1-N} and \eqref{long-eq}, we have
\begin{eqnarray}\label{add-sum-1-N-long-eq}
\begin{aligned}
& \frac{\gamma}{2}\sum\limits_{i=1}^{N-1}\left(\left\| \sum\limits_{j=1}^{i}A_{j}x_{j}^*+\sum\limits_{j=i+1}^{N}A_{j}x_{j}^{k}-b\right\|^{2}-\left\| \sum\limits_{j=1}^{i}A_{j}x_{j}^*+\sum\limits_{j=i+1}^{N}A_{j}x_{j}^{k+1}-b\right\|^{2}\right)\\
& +\frac{1}{\gamma}(\lambda^k-\lambda^{k+1})^{\top}(\lambda^{k+1}-\lambda^*) + \frac{1}{2\gamma}\|\lambda^{k+1}-\lambda^{k}\|^{2} + \frac{\gamma}{2}\sum\limits_{i=2}^{N-1}\left\| \sum\limits_{j=1}^{i-1}A_{j}x_{j}^*+\sum\limits_{j=i}^{N}A_{j}x_{j}^{k+1}-b\right\|^{2} \\
 \geq & \sum\limits_{i=1}^{N}\sigma_i \|x_{i}^{k+1}-x_{i}^*\|^2 + \frac{\gamma}{2}\left\|A_1x_1^{k+1}+\sum_{j=2}^N A_j x_j^k-b\right\|^2.
\end{aligned}
\end{eqnarray}

Using again \eqref{kkt-2}, we obtain
\begin{eqnarray*}
\left\| \sum\limits_{j=1}^{i-1}A_{j}x_{j}^*+\sum\limits_{j=i}^{N}A_{j}x_{j}^{k+1}-b \right\|^{2} = \left\| \sum\limits_{j=i}^{N}A_{j}(x_{j}^{k+1}-x_{j}^{*})\right\|^{2}\leq (N-i+1)\sum\limits_{j=i}^{N}\lambda_{\max}(A_{j}^{\top}A_{j})\| x_{j}^{k+1}-x_{j}^{*}\|^{2},
\end{eqnarray*}
where the inequality follows from the convexity of $\|\cdot\|^2$.
Therefore, we have
\begin{eqnarray}\label{eigen-bound}
\begin{aligned}
& \sum\limits_{i=2}^{N-1}\left\| \sum\limits_{j=1}^{i-1}A_{j}x_{j}^*+\sum\limits_{j=i}^{N}A_{j}x_{j}^{k+1}-b\right\|^{2} \\
\leq & \sum\limits_{i=2}^{N-1}\left((N-i+1)\sum\limits_{j=i}^{N}\lambda_{\max}(A_{j}^{\top}A_{j})\| x_{j}^{k+1}-x_{j}^{*}\|^{2}\right)\\
= & \sum\limits_{i=2}^{N-1}\frac{(2N-i)(i-1)}{2}\lambda_{\max}(A_{i}^{\top}A_{i})\| x_{i}^{k+1}-x_{i}^{*}\|^{2}+\frac{(N+1)(N-2)}{2}\lambda_{\max}(A_{N}^{\top}A_{N})\| x_{N}^{k+1}-x_{N}^{*}\|^{2}.
\end{aligned}
\end{eqnarray}

By combining \eqref{add-sum-1-N-long-eq} and \eqref{eigen-bound} and using the identity
\begin{eqnarray*}
\frac{1}{\gamma}\left(\lambda^{k}-\lambda^{k+1}\right)^{\top}(\lambda^{k+1}-\lambda^*)+\frac{1}{2\gamma}\|\lambda^{k+1}-\lambda^{k}\|^{2} = \frac{1}{2\gamma}\left(\|\lambda^*-\lambda^{k}\|^{2}-\|\lambda^*-\lambda^{k+1}\|^{2}\right),
\end{eqnarray*}
we have
\begin{eqnarray*}
\begin{aligned}
& \frac{\gamma}{2}\sum\limits_{i=1}^{N-1}\left(\left\| \sum\limits_{j=1}^{i}A_{j}x_{j}^*+\sum\limits_{j=i+1}^{N}A_{j}x_{j}^{k}-b\right\|^{2}-\left\| \sum\limits_{j=1}^{i}A_{j}x_{j}^*+\sum\limits_{j=i+1}^{N}A_{j}x_{j}^{k+1}-b\right\|^{2}\right)\\
& +\frac{1}{2\gamma}\left(\|\lambda^*-\lambda^{k}\|^{2}-\|\lambda^*-\lambda^{k+1}\|^{2}\right)\\
\geq & \sum\limits_{i=1}^{N-1}\left[\left(\sigma_i - \frac{\gamma(2N-i)(i-1)}{4}\lambda_{\max}(A_{i}^{\top}A_{i})\right)\| x_{i}^{k+1}-x_{i}^{*}\|^{2}\right]\\
& +\left(\sigma_N-\frac{\gamma(N+1)(N-2)}{4}\lambda_{\max}(A_{N}^{\top}A_{N})\right)\| x_{N}^{k+1}-x_{N}^{*}\|^{2}+ \frac{\gamma}{2}\left\|A_1x_1^{k+1}+\sum_{j=2}^N A_j x_j^k-b\right\|^2,
\end{aligned}
\end{eqnarray*}
which further implies \eqref{lemma-inequality} by using \eqref{kkt-2}.
\end{proof}

\begin{remark}
We note here that \eqref{lemma-inequality} can be equivalently rearranged as
\begin{eqnarray}\label{lemma-inequality-rewritten}
\begin{aligned}
& \left(\gamma\sum\limits_{i=1}^{N-1}\left\| \sum\limits_{j=i+1}^{N}A_{j}(x_{j}^{*}-x_{j}^{k})\right\|^{2}+\frac{1}{2\gamma}\|\lambda^*-\lambda^{k}\|^{2}\right)-\left(\gamma\sum\limits_{i=1}^{N-1}\left\| \sum\limits_{j=i+1}^{N}A_{j}(x_{j}^{*}-x_{j}^{k+1})\right\|^{2} +\frac{1}{2\gamma}\|\lambda^*-\lambda^{k+1}\|^{2}\right)\\
\geq & \sum\limits_{i=1}^{N-1}\left[\left(\sigma_i - \frac{\gamma(2N-i)(i-1)}{4}\lambda_{\max}(A_{i}^{\top}A_{i})\right)\| x_{i}^{k+1}-x_{i}^{*}\|^{2}\right]\\
& +\left(\sigma_N-\frac{\gamma(N+1)(N-2)}{4}\lambda_{\max}(A_{N}^{\top}A_{N})\right)\| x_{N}^{k+1}-x_{N}^{*}\|^{2}+ \frac{\gamma}{2}\left\|A_1x_1^{k+1}+\sum_{j={2}}^N A_j x_j^k-b\right\|^2\\
& + \frac{\gamma}{2}\sum\limits_{i=1}^{N-1}\left\| \sum\limits_{j=i+1}^{N}A_{j}(x_{j}^{*}-x_{j}^{k})\right\|^{2} - \frac{\gamma}{2}\sum\limits_{i=1}^{N-1}\left\| \sum\limits_{j=i+1}^{N}A_{j}(x_{j}^{*}-x_{j}^{k+1})\right\|^{2}.
\end{aligned}
\end{eqnarray}
Both \eqref{lemma-inequality} and \eqref{lemma-inequality-rewritten} will be used in subsequent analysis. In scenario 1, we will use \eqref{lemma-inequality} to show that \[ \frac{\gamma}{2}\sum\limits_{i=1}^{N-1}\left\| \sum\limits_{j=i+1}^{N}A_{j}(x_{j}^{*}-x_{j}^{k})\right\|^{2}+\frac{1}{2\gamma}\|\lambda^*-\lambda^{k}\|^{2}   \] converges to zero linearly; in scenarios 2 and 3, we will use \eqref{lemma-inequality-rewritten} to show that
\[ \gamma\sum\limits_{i=1}^{N-1}\left\| \sum\limits_{j=i+1}^{N}A_{j}(x_{j}^{*}-x_{j}^{k})\right\|^{2}+\frac{1}{2\gamma}\|\lambda^*-\lambda^{k}\|^{2} \]
converges to zero linearly.
\end{remark}

The next lemma considers the convergence of $\{x_1^k,\ldots,x_N^k,\lambda^k\}$ under conditions listed in scenarios 2 and 3 in Table \ref{tab:3-scenarios}.
\begin{lemma}\label{lem-convergence}
Assume that the conditions listed in scenario 2 or scenario 3 in Table \ref{tab:3-scenarios} hold. Moreover, we assume that $\gamma$ satisfies the following conditions:
\be\label{gamma-def}
\gamma < \min_{i=2,\ldots, N-1}\left\{\frac{4\sigma_{i}}{(2N-i)(i-1)\lambda_{\max}(A_{i}^\top A_{i})}, \frac{4\sigma_{N}}{(N+1)(N-2)\lambda_{\max}(A_{N}^\top A_{N})}\right\}.
\ee
Then
$(x_1^k,\ldots,x_N^k,\lambda^k)$ generated by ADMM \eqref{admm-N} converges to some $({x}_1^*,\ldots,{x}_N^*,{\lambda}^*)\in\Omega^*$.
\end{lemma}

\begin{proof}
Note that the conditions listed in scenarios 2 and 3 in Table \ref{tab:3-scenarios} both require that $f_2,\ldots,f_N$ are strongly convex. Denote the right hand side of inequality \eqref{lemma-inequality} by $\xi^k$. It follows from \eqref{gamma-def} and \eqref{lemma-inequality} that $\xi^k\geq 0$ and $\sum_{k=0}^\infty \xi^k <+\infty$, which further implies that $\xi^k\rightarrow 0$. Hence, for any $({x}_1^*,\ldots,{x}_N^*,{\lambda}^*)\in\Omega^*$, we have $x_i^k-x_i^*\rightarrow 0$ for $i=2,\ldots,N$, and $A_1x_1^{k+1}+\sum_{j=2}^N A_j x_j^k-b\rightarrow 0$, which also implies that $A_1 x_1^k - A_1 x_1^* \rightarrow 0$. In scenario 2, it is assumed that $f_1$ is strongly convex. Thus $\sigma_1>0$ and \eqref{lemma-inequality} implies that $x_1^k-x_1^*\rightarrow 0$. In scenario 3, it is assumed that $A_1$ is of full column rank. It thus follows from $A_1 x_1^k - A_1 x_1^* \rightarrow 0$ that $x_1^k-x_1^*\rightarrow 0$.

Moreover, when \eqref{gamma-def} holds, it follows from \eqref{lemma-inequality} that $\frac{\gamma}{2}\sum_{i=1}^{N-1}\|\sum_{j=i+1}^N A_j(x_j^*-x_j^k)\|^2 + \frac{1}{2\gamma}\|\lambda^*-\lambda^k\|^2$ is non-increasing and upper bounded. It thus follows that $\|\lambda^*-\lambda^k\|^2$ converges and $\{\lambda^k\}$ is bounded. Therefore, $\{\lambda^k\}$ has a converging subsequence $\{\lambda^{k_j}\}$. Let $\bar{\lambda}=\lim_{j\rightarrow\infty}\{\lambda^{k_j}\}$. By passing the limit in \eqref{opt-x-lambda}, it holds that $A_i^\top \bar{\lambda} = \nabla f_i(x_i^*)$ for $i=1,2,\ldots,N$. Thus, $({x}_1^*,\ldots,{x}_N^*,\bar{\lambda})\in\Omega^*$ and we can just let $\lambda^*=\bar{\lambda}$. Since $\|\lambda^*-\lambda^k\|^2$ converges and $\lambda^{k_j}\rightarrow \lambda^*$, we conclude that $\lambda^k\rightarrow \lambda^*$.
\end{proof}

Before proceeding to the next lemma, we define a constant $\kappa $ that will be used subsequently.
\begin{definition}\label{def:kappa}
We define a constant $\kappa $ as follows.
\begin{itemize}
\item (i). If the matrix $[A_{1},\cdots,A_{N}]$ is of full row rank, then $\kappa :=\lambda_{\min}^{-1}([A_{1},\cdots,A_{N}][A_{1},\cdots,A_{N}]^{\top})>0$.
\item (ii). Otherwise, assume $\rank([A_{1},\cdots,A_{N}])=r<p$. Without loss of generality, assuming that the first $r$ rows of $[A_{1},\cdots,A_{N}]$ (denoted by $[A_{1}^r,\cdots,A_{N}^r]$) are linearly independent, we have
\be\label{lemma3-I-L-r}
[A_{1},\cdots,A_{N}]=\left[\begin{array}{c} I \\ B \end{array}\right][A_{1}^r,\cdots,A_{N}^r],
\ee
where $I\in\RR^{r\times r}$ is the identity matrix and $B\in\RR^{(p-r)\times r}$. Let $E:=(I+B^{\top}B)[A_{1}^r,\cdots,A_{N}^r]$. It is easy to see that $E$ has full row rank. Then $\kappa $ is defined as  $\kappa :=\lambda_{\min}^{-1}(EE^{\top})\lambda_{\max}(I+B^{\top}B)>0$.
\end{itemize}
\end{definition}

The next lemma concerns bounding $\|\lambda^{k+1}-\lambda^*\|^2$ using terms related to $x_i^{k}-x_i^*$, $i=1,\ldots,N$.

\begin{lemma}\label{lemma3}
Let $(x_1^*,\ldots,x_N^*,\lambda^*)\in\Omega^*$. Assume that the conditions listed in scenario 2 or scenario 3 in Table \ref{tab:3-scenarios} hold, and $\gamma$ satisfies \eqref{gamma-def}.
Suppose $\nabla f_{i}$ is Lipschitz continuous with constant $L_{i}$ for $i=1,\ldots,N$, and the initial Lagrange multiplier $\lambda^0$ is in the range space of $[A_{1},\cdots,A_{N}]$ (note that letting $\lambda^0=0$ suffices). It holds that
\begin{eqnarray}\label{lemma3-inequality}
\begin{aligned}
\| \lambda^{k+1}-\lambda^*\|^{2} \leq &  \sum\limits_{i=1}^{N} \left(2\kappa  L_{i}^2\right)\| x_{i}^{k+1} - x_{i}^{*} \|^2 \\ &  + \sum\limits_{i=1}^{N-1}\left(4\kappa \gamma^2\lambda_{\max}(A_{i}^{\top}A_{i})\right)\left(\left\|\sum\limits_{j=i+1}^{N}A_{j}(x_{j}^{k}-x_{j}^{*})\right\|^2 + \left\|\sum\limits_{j=i+1}^{N}A_{j}(x_{j}^{k+1}-x_{j}^{*})\right\|^2\right),
\end{aligned}
\end{eqnarray}
where $\kappa  > 0$ is defined in Definition \ref{def:kappa}.
\end{lemma}
\begin{proof}
We first show the following inequality
\begin{eqnarray}\label{lambda-cbar}
\| \lambda^{k+1}-\lambda^*\|^{2}\leq\kappa \cdot\left\|\left[\begin{array}{c} A_{1}^{\top}\\ \vdots \\ A_{N}^{\top} \end{array}\right](\lambda^{k+1}-\lambda^*)\right\|^{2}.
\end{eqnarray}
In case (i), $[A_{1},\cdots,A_{N}]$ has full row rank, so \eqref{lambda-cbar} holds trivially. Now we consider case (ii). By the updating formula of $\lambda^{k+1}$ \eqref{update-lambda} and \eqref{kkt-2}, we know that if the initial Lagrange multiplier $\lambda^0$ is in the range space of $[A_{1},\cdots,A_{N}]$, then $\lambda^k, k=1,2,\cdots$, always stay in the range space of $[A_{1},\cdots,A_{N}]$, so does $\lambda^*$. Therefore, from \eqref{lemma3-I-L-r}, we can get
\begin{displaymath}
\lambda^{k+1}=\left[\begin{array}{c} I \\ B \end{array}\right]\lambda_{r}^{k+1}, \quad
\lambda^*=\left[\begin{array}{c} I \\ B \end{array}\right]\lambda_{r}^*, \quad
\left[\begin{array}{c} A_{1}^{\top} \\ \vdots \\ A_{N}^{\top} \end{array}\right](\lambda^{k+1}-\lambda^*) = \left[\begin{array}{c} (A_{1}^r)^{\top} \\ \vdots \\ (A_{N}^r)^{\top} \end{array}\right](I+B^{\top}B)(\lambda_{r}^{k+1}-\lambda_{r}^*),
\end{displaymath}
where $\lambda_r^{k+1}$ and $\lambda_r^*$ denote the first $r$ rows of $\lambda^{k+1}$ and $\lambda^*$, respectively.
Since $E:=(I+B^{\top}B)[A_{1}^r,\cdots,A_{N}^r]$ has full row rank, it now follows that
\[\left\|\left[ \ba{c} A_1^\top \\ \vdots \\ A_N^\top \ea \right] (\lambda^{k+1}-\lambda^*)\right\|^2 = \|E^\top (\lambda_r^{k+1}-\lambda_r^*)\|^2 \geq \lambda_{\min}(EE^\top)\|\lambda_r^{k+1}-\lambda_r^*\|^2 \geq \frac{\lambda_{\min}(EE^\top)}{\lambda_{\max}(I+B^\top B)}\|\lambda^{k+1}-\lambda^*\|^2, \]
which implies \eqref{lambda-cbar}.

Using the optimality conditions \eqref{opt-x-lambda}, and the Lipschitz continuity of $\nabla f_{i}, i=1,\cdots,N$, we have
\begin{eqnarray*}
\begin{aligned}
& \left\|\left[\begin{array}{c} A_{1}^{\top} \\ A_{2}^{\top} \\ \vdots \\ A_{N}^{\top}\end{array}\right](\lambda^{k+1}-\lambda^*)+\left[\begin{array}{c} -\gamma A_{1}^{\top} \\ 0 \\ \vdots \\ 0 \end{array}\right]\left(\sum\limits_{j=2}^{N}A_{j}(x_{j}^{k}-x_{j}^{k+1})\right)+\cdots + \left[\begin{array}{c} 0 \\ \vdots \\ -\gamma A_{N-1}^{\top} \\ 0 \end{array}\right](A_{N}(x_{N}^{k}-x_{N}^{k+1}))\right\|^2 \\
= & \sum\limits_{i=1}^{N}\| \nabla f_{i}(x_{i}^{k+1})- \nabla f_{i}(x_{i}^{*})\|^2\leq\sum\limits_{i=1}^{N} L_{i}^2\| x_{i}^{k+1} - x_{i}^{*} \|^2,
\end{aligned}
\end{eqnarray*}
which together with \eqref{lambda-cbar} implies that
\begin{eqnarray*}
\begin{aligned}
& \| \lambda^{k+1}-\lambda^*\|^{2} \\ \leq & \kappa \cdot\left\|\left[\begin{array}{c} A_{1}^{\top}\\ A_{2}^{\top} \\ \vdots \\ A_{N}^{\top} \end{array}\right](\lambda^{k+1}-\lambda^*)\right\|^{2}\\
\leq & 2 \kappa  \left(\left\|\left[\begin{array}{c} -\gamma A_{1}^{\top} \\ 0 \\ \vdots \\ 0 \end{array}\right]\left(\sum\limits_{j=2}^{N}A_{j}(x_{j}^{k}-x_{j}^{k+1})\right)+\cdots + \left[\begin{array}{c} 0 \\ \vdots \\ -\gamma A_{N-1}^{\top} \\ 0 \end{array}\right](A_{N}(x_{N}^{k}-x_{N}^{k+1}))\right\|^2 + \sum\limits_{i=1}^{N} L_{i}^2\| x_{i}^{k+1} - x_{i}^{*} \|^2\right)\\
\leq & 2\kappa \gamma^2\sum\limits_{i=1}^{N-1}\lambda_{\max}(A_{i}^{\top}A_{i})\left\|\sum\limits_{j=i+1}^{N}A_{j}(x_{j}^{k}-x_{j}^{k+1})\right\|^2 + 2\kappa \sum\limits_{i=1}^{N} L_{i}^2\| x_{i}^{k+1} - x_{i}^{*} \|^2\\
\leq & 4\kappa \gamma^2\sum\limits_{i=1}^{N-1}\lambda_{\max}(A_{i}^{\top}A_{i})\left(\left\|\sum\limits_{j=i+1}^{N}A_{j}(x_{j}^{k}-x_{j}^{*})\right\|^2 + \left\|\sum\limits_{j=i+1}^{N}A_{j}(x_{j}^{k+1}-x_{j}^{*})\right\|^2\right) + 2\kappa \sum\limits_{i=1}^{N} L_{i}^2\| x_{i}^{k+1} - x_{i}^{*} \|^2.
\end{aligned}
\end{eqnarray*}
\end{proof}

\section{Global Linear Convergence of the ADMM}\label{sec:lin-con}

In this section, we prove the global linear convergence of the ADMM \eqref{admm-N} under the three scenarios listed in Table \ref{tab:3-scenarios}.
We note the following inequality,
\be\label{convex-L2}
\sum\limits_{i=1}^{N-1}\left\| \sum\limits_{j=i+1}^{N}A_{j}(x_{j}^{*}-x_{j}^{k+1})\right\|^{2}
\leq  \sum\limits_{i=2}^{N}\left[\frac{(2N-i)(i-1)}{2}\lambda_{\max}(A_{i}^\top A_{i})\| x_{i}^{*}-x_{i}^{k+1}\|^{2}\right],
\ee
which follows from the convexity of $\|\cdot\|^2$. We shall use this inequality in our subsequent analysis.

\subsection{$Q$-linear convergence under scenario 1}

\begin{theorem}\label{theorem:main}
Suppose that the conditions listed in scenario 1 in Table \ref{tab:3-scenarios} hold. If $\gamma$ satisfies \eqref{gamma-def},
then it holds that
\begin{eqnarray}\label{theorem-inequality}
\begin{aligned}
& \left(\frac{\gamma}{2}\sum\limits_{i=1}^{N-1}\left\| \sum\limits_{j=i+1}^{N}A_{j}(x_{j}^{*}-x_{j}^{k})\right\|^{2}+\frac{1}{2\gamma}\|\lambda^*-\lambda^{k}\|^{2}\right)\\
\geq & (1+\delta_1)\left(\frac{\gamma}{2}\sum\limits_{i=1}^{N-1}\left\| \sum\limits_{j=i+1}^{N}A_{j}(x_{j}^{*}-x_{j}^{k+1})\right\|^{2}+\frac{1}{2\gamma}\|\lambda^*-\lambda^{k+1}\|^{2}\right),
\end{aligned}
\end{eqnarray}
where
\be\label{delta-def}
\delta_1 := \min_{i=2,\ldots, N-1}\left\{\frac{4\sigma_i - \gamma(2N-i)(i-1)\lambda_{\max}(A_{i}^{\top}A_{i})}{\gamma(2N-i)(i-1)\lambda_{\max}(A_{i}^\top A_{i})}, \frac{4\gamma\sigma_N- \gamma^2(N+1)(N-2)\lambda_{\max}(A_{N}^{\top}A_{N})}{2\lambda_{\min}^{-1}(A_{N}A_{N}^\top)L_N^{2}+\gamma^2 N(N-1)\lambda_{\max}(A_{N}^\top A_{N})}\right\}.
\ee
Note that it follows from \eqref{gamma-def} that $\delta_1 > 0$.  As a result of \eqref{theorem-inequality}, we conclude that
\[\left(\sum\limits_{j=2}^{N}A_{j}x_{j}^{k}, \sum\limits_{j=3}^{N}A_{j}x_{j}^{k}, \ldots, \sum\limits_{j=N}^{N}A_{j}x_{j}^{k}, \lambda^k\right)\]
converges $Q$-linearly.
\end{theorem}

\begin{proof}
Because $\nabla f_N$ is Lipschitz continuous with constant $L_N$, by setting $i=N$ in \eqref{opt-x-lambda} and \eqref{kkt-1}, we get
\begin{eqnarray*}
\| A_{N}^\top (\lambda^{k+1}-\lambda^*)\|^{2} = \| \nabla f_{N}(x_{N}^{k+1})-\nabla f_{N}(x_{N}^{*})\|^{2}\leq L_N^{2}\| x_{N}^{k+1}-x_{N}^*\|^{2},
\end{eqnarray*}
which implies
\begin{eqnarray}\label{AN-bound}
\| \lambda^{k+1}-\lambda^*\|^{2}\leq\lambda_{\min}^{-1}(A_{N}A_{N}^\top)L_N^{2}\| x_{N}^{k+1}-x_{N}^*\|^{2},
\end{eqnarray}
due to the fact that $A_N$ is of full row rank.

By combining \eqref{lemma-inequality}, \eqref{delta-def}, \eqref{convex-L2} and \eqref{AN-bound}, it follows that (note that we do not assume that $f_1$ is strongly convex, and thus $\sigma_1=0$),
\begin{eqnarray*}
\begin{aligned}
& \left(\frac{\gamma}{2}\sum\limits_{i=1}^{N-1}\left\| \sum\limits_{j=i+1}^{N}A_{j}(x_{j}^{*}-x_{j}^{k})\right\|^{2}+\frac{1}{2\gamma}\|\lambda^*-\lambda^{k}\|^{2}\right)  -\left(\frac{\gamma}{2}\sum\limits_{i=1}^{N-1}\left\| \sum\limits_{j=i+1}^{N}A_{j}(x_{j}^{*}-x_{j}^{k+1})\right\|^{2} +\frac{1}{2\gamma}\|\lambda^*-\lambda^{k+1}\|^{2}\right)\\
\geq & \sum\limits_{i=2}^{N-1}\left[\left(\sigma_i - \frac{\gamma(2N-i)(i-1)}{4}\lambda_{\max}(A_{i}^{\top}A_{i})\right)\| x_{i}^{k+1}-x_{i}^{*}\|^{2}\right]\\
& +\left(\sigma_N-\frac{\gamma(N+1)(N-2)}{4}\lambda_{\max}(A_{N}^{\top}A_{N})\right)\| x_{N}^{k+1}-x_{N}^{*}\|^{2} \\
\geq & \delta_1 \left[  \sum\limits_{i=2}^{N}\left[\frac{\gamma(2N-i)(i-1)}{4}\lambda_{\max}(A_{i}^\top A_{i})\| x_{i}^{*}-x_{i}^{k+1}\|^{2}\right] +\frac{\lambda_{\min}^{-1}(A_{N}A_{N}^\top)L_N^{2}}{2\gamma}\| x_{N}^{*}-x_{N}^{k+1}\|^{2}\right] \\
\geq & \delta_1 \left[\frac{\gamma}{2}\sum\limits_{i=1}^{N-1}\left\| \sum\limits_{j=i+1}^{N}A_{j}(x_{j}^{*}-x_{j}^{k+1})\right\|^{2}+\frac{1}{2\gamma}\|\lambda^*-\lambda^{k+1}\|^{2}\right],
\end{aligned}
\end{eqnarray*}
which further implies \eqref{theorem-inequality}.
\end{proof}

\subsection{$Q$-linear convergence under scenario 2}

\begin{theorem}\label{thm2}
Suppose that the conditions listed in scenario 2 in Table \ref{tab:3-scenarios} hold. If $\gamma$ satisfies
\be\label{gamma-def-2}
\gamma < \min_{i=2,\cdots, N-1}\left\{\frac{4\sigma_{i}}{3(2N-i)(i-1)\lambda_{\max}(A_{i}^\top A_{i})}, \frac{4\sigma_{N}}{(3N^2-3N-2)\lambda_{\max}(A_{N}^\top A_{N})}\right\},
\ee
then it holds that
\begin{eqnarray}\label{thm2-inequality}
\begin{aligned}
& \left(\gamma\sum\limits_{i=1}^{N-1}\left\| \sum\limits_{j=i+1}^{N}A_{j}(x_{j}^{*}-x_{j}^{k})\right\|^{2}+\frac{1}{2\gamma}\|\lambda^*-\lambda^{k}\|^{2}\right)\\
\geq & (1+\delta_2)\left(\gamma\sum\limits_{i=1}^{N-1}\left\| \sum\limits_{j=i+1}^{N}A_{j}(x_{j}^{*}-x_{j}^{k+1})\right\|^{2}+\frac{1}{2\gamma}\|\lambda^*-\lambda^{k+1}\|^{2}\right),
\end{aligned}
\end{eqnarray}
where
\be\label{def-delta-2}
\delta_2 := \min\left\{\frac{\sigma_1\gamma}{\kappa L_{1}^2}, \delta_3, \delta_4, \delta_5\right\},
\ee
and
\begin{eqnarray}\label{delta-456}
\left\{
\begin{aligned}
\delta_3   := & \min_{i=2,\ldots,N-1}\left\{\frac{4\sigma_i\gamma - 3\gamma^2 (2N-i)(i-1)\lambda_{\max}(A_i^\top A_i)}{2\gamma^2(2N-i)(i-1)\lambda_{\max}(A_i^\top A_i) + 4\kappa L_i^2}\right\}, \\
\delta_4   := & \min_{i=1,\ldots,N-1}\left\{\frac{1}{4\kappa \lambda_{\max}(A_{i}^{\top}A_{i})}\right\}, \\
\delta_5   := & \frac{4\sigma_N\gamma-(3N^2-3N-2)\gamma^2\lambda_{\max}(A_N^\top A_N)}{2\gamma^2 N(N-1)\lambda_{\max}(A_N^\top A_N)+4\kappa L_N^2},
\end{aligned}
\right.
\end{eqnarray}
where $\kappa $ is defined in Definition \ref{def:kappa}.
Note that it follows from \eqref{gamma-def-2} that $\delta_2 > 0$. As a result of \eqref{thm2-inequality}, we conclude that
\[\left(\sum_{j=2}^N A_jx_j^k,\sum_{j=3}^N A_jx_j^k, \ldots,\sum_{j=N}^N A_jx_j^k,\lambda^k\right) \]
converges $Q$-linearly.
\end{theorem}

\begin{proof}
By combining \eqref{lemma3-inequality} and \eqref{convex-L2}, we have
\begin{eqnarray}\label{2_2}
\begin{aligned}
& (1+\delta_2)\left(\gamma\sum\limits_{i=1}^{N-1}\left\| \sum\limits_{j=i+1}^{N}A_{j}(x_{j}^{*}-x_{j}^{k+1})\right\|^{2}\right)+\delta_2\left(\frac{1}{2\gamma}\|\lambda^*-\lambda^{k+1}\|^{2}\right) \\
\leq & (1+\delta_2)\sum\limits_{i=2}^{N}\left[\frac{\gamma(2N-i)(i-1)}{2}\lambda_{\max}(A_{i}^\top A_{i})\| x_{i}^{*}-x_{i}^{k+1}\|^{2}\right] \\ & + \frac{\delta_2}{2\gamma}\left(\sum\limits_{i=1}^{N-1}\left(4\kappa \gamma^2\lambda_{\max}(A_{i}^{\top}A_{i})\right)\left(\left\|\sum\limits_{j=i+1}^{N}A_{j}(x_{j}^{k}-x_{j}^{*})\right\|^2 + \left\|\sum\limits_{j=i+1}^{N}A_{j}(x_{j}^{k+1}-x_{j}^{*})\right\|^2\right)\right. \\ & \left.+ \sum\limits_{i=1}^{N} \left(2\kappa L_{i}^2\right)\| x_{i}^{k+1} - x_{i}^{*} \|^2 \right)\\
\leq & \sum\limits_{i=1}^{N-1}\left[\left(\sigma_i - \frac{\gamma(2N-i)(i-1)}{4}\lambda_{\max}(A_{i}^{\top}A_{i})\right)\| x_{i}^{k+1}-x_{i}^{*}\|^{2}\right]\\
& +\left(\sigma_N-\frac{\gamma(N+1)(N-2)}{4}\lambda_{\max}(A_{N}^{\top}A_{N})\right)\| x_{N}^{k+1}-x_{N}^{*}\|^{2}\\
& + \frac{\gamma}{2}\sum\limits_{i=1}^{N-1}\left\| \sum\limits_{j=i+1}^{N}A_{j}(x_{j}^{*}-x_{j}^{k})\right\|^{2} + \frac{\gamma}{2}\sum\limits_{i=1}^{N-1}\left\| \sum\limits_{j=i+1}^{N}A_{j}(x_{j}^{*}-x_{j}^{k+1})\right\|^{2},
\end{aligned}
\end{eqnarray}
where the last inequality follows from the definition of $\delta_2$ in \eqref{def-delta-2}.
Finally we note that combining \eqref{2_2} with \eqref{lemma-inequality-rewritten} yields \eqref{thm2-inequality}.
\end{proof}

\subsection{$Q$-linear convergence under scenario 3}

\begin{theorem}\label{thm3}
Suppose that the conditions listed in scenario 3 in Table \ref{tab:3-scenarios} hold. If $\gamma$ satisfies \eqref{gamma-def-2},
then it holds that
\begin{eqnarray}\label{thm3-inequality}
\begin{aligned}
& \left(\gamma\sum\limits_{i=1}^{N-1}\left\| \sum\limits_{j=i+1}^{N}A_{j}(x_{j}^{*}-x_{j}^{k})\right\|^{2}+\frac{1}{2\gamma}\|\lambda^*-\lambda^{k}\|^{2}\right)\\
\geq & (1+\delta_6)\left(\gamma\sum\limits_{i=1}^{N-1}\left\| \sum\limits_{j=i+1}^{N}A_{j}(x_{j}^{*}-x_{j}^{k+1})\right\|^{2}+\frac{1}{2\gamma}\|\lambda^*-\lambda^{k+1}\|^{2}\right),
\end{aligned}
\end{eqnarray}
where
\be\label{def-delta-3}
\delta_6 := \min\left\{\frac{\gamma^2}{4\kappa \gamma^2(N-1)\lambda_{\max}(A_{1}^{\top}A_{1})+ 4\kappa L_{1}^2\lambda_{\min}^{-1}(A_{1}^{\top}A_{1})}, \delta_3,\delta_4,\delta_5\right\},
\ee
with $\delta_3$, $\delta_4$ and $\delta_5$ defined in \eqref{delta-456}.
Note that it follows from \eqref{gamma-def-2} that $\delta_6 > 0$. As a result of \eqref{thm3-inequality}, we conclude that
\[\left(\sum_{j=2}^NA_jx_j^k,\sum_{j=3}^NA_jx_j^k,\ldots,\sum_{j=N}^NA_jx_j^k,\lambda^k \right)\]
converges $Q$-linearly.
\end{theorem}

\begin{proof}
Since $A_1$ is of full column rank, it is easy to verify that
\begin{eqnarray}\label{lemma4-inequality}
\begin{aligned}
\lambda_{\min}(A_1^\top A_1) \|x_{1}^{k+1}-x_{1}^{*}\|^2 \leq & \| A_{1}(x_{1}^{k+1}-x_{1}^{*})\|^{2} \\
= & \left\| \left( A_{1}x_{1}^{k+1}+\sum\limits_{j=2}^{N}A_{j}x_{j}^{k}-b\right)-\left(\sum\limits_{j=2}^{N}A_{j}(x_{j}^{k}-x_{j}^{*})\right)\right\|^2\\
\leq & 2 \left\| A_{1}x_{1}^{k+1}+\sum\limits_{j=2}^{N}A_{j}x_{j}^{k}-b\right\|^2 + 2\left\| \sum\limits_{j=2}^{N}A_{j}(x_{j}^{k}-x_{j}^{*})\right\|^2.
\end{aligned}
\end{eqnarray}

Combining \eqref{lemma4-inequality} and \eqref{lemma3-inequality} yields,
\begin{eqnarray}\label{3_2}
\begin{aligned}
& \frac{1}{2\gamma}\|\lambda^*-\lambda^{k+1}\|^{2}\\
\leq &  \sum\limits_{i=2}^{N} \left(\frac{\kappa L_{i}^2}{\gamma}\right)\| x_{i}^{*} - x_{i}^{k+1} \|^2\\
& + \frac{2}{\gamma}(\kappa L_{1}^2 \lambda_{\min}^{-1}(A_{1}^{\top}A_{1}))\left( \left\| A_{1}x_{1}^{k+1}+\sum\limits_{j=2}^{N}A_{j}x_{j}^{k}-b\right\|^2 + \left\| \sum\limits_{j=2}^{N}A_{j}(x_{j}^{k}-x_{j}^{*})\right\|^2\right)\\
& + \sum\limits_{i=1}^{N-1}\left(2\kappa \gamma(N-1)\lambda_{\max}(A_{i}^{\top}A_{i})\right)\left(\left\|\sum\limits_{j=i+1}^{N}A_{j}(x_{j}^{k}-x_{j}^{*})\right\|^2 + \left\|\sum\limits_{j=i+1}^{N}A_{j}(x_{j}^{k+1}-x_{j}^{*})\right\|^2 \right).
\end{aligned}
\end{eqnarray}
Combining \eqref{3_2}, \eqref{convex-L2} and \eqref{def-delta-3} yields,
\begin{eqnarray*}
\begin{aligned}
& (1+\delta_6)\gamma\sum\limits_{i=1}^{N-1}\left\| \sum\limits_{j=i+1}^{N}A_{j}(x_{j}^{*}-x_{j}^{k+1})\right\|^{2} + \delta_6\frac{1}{2\gamma}\|\lambda^*-\lambda^{k+1}\|^{2} \\
\leq & \sum\limits_{i=1}^{N-1}\left[\left(\sigma_i - \frac{\gamma(2N-i)(i-1)}{4}\lambda_{\max}(A_{i}^{\top}A_{i})\right)\| x_{i}^{k+1}-x_{i}^{*}\|^{2}\right]\\
& +\left(\sigma_N-\frac{\gamma(N+1)(N-2)}{4}\lambda_{\max}(A_{N}^{\top}A_{N})\right)\| x_{N}^{k+1}-x_{N}^{*}\|^{2}\\
& + \frac{\gamma}{2} \left\|A_{1}x_{1}^{k+1}+\sum\limits_{j=2}^{N}A_{j}x_{j}^{k}-b\right\|^{2}\\
& + \frac{\gamma}{2}\sum\limits_{i=1}^{N-1}\left\| \sum\limits_{j=i+1}^{N}A_{j}(x_{j}^{*}-x_{j}^{k})\right\|^{2} + \frac{\gamma}{2}\sum\limits_{i=1}^{N-1}\left\| \sum\limits_{j=i+1}^{N}A_{j}(x_{j}^{*}-x_{j}^{k+1})\right\|^{2},
\end{aligned}
\end{eqnarray*}
which together with \eqref{lemma-inequality-rewritten} implies \eqref{thm3-inequality}.
\end{proof}

\subsection{$R$-linear Convergence}

From the results in Theorems \ref{theorem:main}, \ref{thm2} and \ref{thm3}, we have the following immediate corollary on the $R$-linear convergence of ADMM \eqref{admm-N}.

\begin{corollary}\label{cor:R-linear}
Under the same conditions in Theorem \ref{theorem:main}, or Theorem \ref{thm2}, or Theorem \ref{thm3}, $x_N^k$, $\lambda^k$ and $A_ix_i^k, i =1,\ldots,N-1$ converge $R$-linearly. Moreover, if $A_{i}, i=1,2,\cdots,N-1$ are further assumed to be of full column rank, then $x_{i}^{k}, i=1,2,\cdots,N-1$
converge $R$-linearly. 
\end{corollary}

\begin{proof}
Note that under all the three scenarios, we have shown that
the sequence
\[\left(\sum_{j=2}^NA_jx_j^k,\sum_{j=3}^NA_jx_j^k,\ldots,\sum_{j=N}^NA_jx_j^k,\lambda^k \right)\]
converges $Q$-linearly. It follows that $\lambda^{k}$ and $\sum\limits_{j=i+1}^{N}A_{j}x_{j}^{k}, i=1,\cdots,N-1$  converge $R$-linearly, since any part of a $Q$-linear convergent quantity converges $R$-linearly. It also implies that $A_2x_2^k,\ldots,A_Nx_N^k$ converge $R$-linearly. It now follows from \eqref{lambda-opt} that $A_{1}x_{1}^{k}$ converges $R$-linearly.
By setting $i = N$ in \eqref{fi-opt}, one obtains,
\[
(x_{N}^{k+1}-x_{N}^*)^{\top}A_{N}^\top(\lambda^{k+1}-\lambda^*) \geq \sigma_N \|x_{N}^{k+1}-x_{N}^*\|^2,
\]
which implies that
\[
\| x_{N}^{k+1}-x_{N}^*\|\| A_{N}\|\|\lambda^{k+1}-\lambda^*\| \geq \sigma_N \|x_{N}^{k+1}-x_{N}^*\|^2,
\]
i.e.,
\[
\|x_{N}^{k+1}-x_{N}^*\| \leq \frac{\| A_{N}\|}{\sigma_{N}}\|\lambda^{k+1}-\lambda^*\|.
\]
The $R$-linear convergence of $x_N^k$ then follows from the fact that $\lambda^k$ converges $R$-linearly.
\end{proof}


{
Now we make some remarks on the convergence results presented in this section.
\begin{remark}
If we incorporate the indicator function into the objective function in \eqref{prob:P}, then its subgradient cannot be Lipschitz continuous on the boundary of the constraint set. Therefore, scenarios 2 and 3 can only occur if the constraint sets $\XCal_i$'s are actually the whole space. However, scenario 1 does allow most of the constraint sets to exist; essentially, it only requires that
$x_N$ is unconstrained, and all other blocks of variables can be constrained.
It remains an interesting question to figure out if the linear convergence rate still holds if all blocks of variables are constrained.

\end{remark}
\begin{remark}
Finally, we remark that the scenario 1 in Table \ref{tab:3-scenarios} also gives rise to a linear convergence rate of the ADMM for convex optimization with {\em inequality}\/ constraints:
\[
\begin{array}{ll}
\min & \tilde{f}_1(x_1) + \tilde{f}_2(x_2) + \cdots + \tilde{f}_N(x_N) \\
\st &  A_1 x_1 + A_2 x_2 + \cdots + A_N x_N \le b \\
& x_i \in {\cal X}_i ,\, i=1,2,\ldots,N.
\end{array}
\]
In that case, by introducing a slack variable $x_0$ with the constraint $x_0 \in \RR^p_+$, the corresponding ADMM becomes
\[
\left\{\ba{lcl}
x_0^{k+1}     & := & \argmin_{x_0 \in \RR^p_+} \, \LCal_\gamma(x_0,x_1^{k},\ldots,x_N^k;\lambda^k) = \left( -\sum_{i=1}^N A_i x^k_i + b + \frac{1}{\gamma} \lambda^k \right)_+ , \\ \\
x_i^{k+1}     & := & \argmin_{x_i \in {\cal X}_i} \, \LCal_\gamma(x_0^{k+1},x_1^{k+1},\ldots,x^{k+1}_{i-1},x_i,x^{k}_{i+1},\cdots, x_N^k;\lambda^k),\, i=1,2,...,N , \\ \\
\lambda^{k+1} & := & \displaystyle\lambda^k - \gamma \left(x^{k+1}_0 + \sum_{i=1}^N A_i x_i^{k+1} -b\right),         \ea\right.
\]
where
\[
\LCal_\gamma(x_0,x_1,\ldots,x_N;\lambda) := \sum_{i=1}^N \tilde{f}_i(x_i) - \left\langle \lambda,x_0+\sum_{i=1}^NA_ix_i-b \right\rangle + \frac{\gamma}{2} \left\|x_0+\sum_{i=1}^NA_ix_i-b \right\|^2 .
\]
Suppose that the functions $\tilde{f}_i$, $i=1,\ldots,N$ are all strongly convex, and $\nabla\tilde{f}_N$ is Lipschitz continuous, $x_N\in\XCal_N$ does not present and $A_N$ has full row rank, Theorem \ref{theorem:main} assures that the above ADMM algorithm converges globally linearly. 
\end{remark}

\subsection{An illustrative example showing the linear convergence of \eqref{admm-N}}

In this section, we implement the ADMM \eqref{admm-N} on a simple example to show its linear convergence under the conditions in scenario 1 in Table \ref{tab:3-scenarios}. The problem under consideration is of the following form:
\be\label{toy-example}\ba{ll} \min & x_2^\top P x_2 + a^\top x_2 + x_3^\top Q x_3 + b^\top x_3 \\
                               \st & x_1 - Ax_2 - x_3 = 0, \\
                                   & x_1 \geq 0, \ea\ee
where $P\in\RR^{n_2\times n_2}$ and $Q\in\RR^{n_3\times n_3}$ are diagonal matrices with all diagonal entries being positive, $A\in\RR^{n_1\times n_2}$. This example is a special case of the problem of utility maximization for smart grid considered in \cite{Samadi-pricing-smart-grid-2010}. It is noted that the problem \eqref{toy-example} satisfies all the conditions given in scenario 1 in Table \ref{tab:3-scenarios}. We now numerically verify that the iterates generated by ADMM \eqref{admm-N} for solving \eqref{toy-example} converge globally linearly. In our experiments, we chose $n_1=n_3=20$, $n_2=50$; the diagonal entries of $P$ and $Q$ were generated randomly following uniform distribution in $[0,1]$; the entries of $A$, $a$ and $b$ were generated following standard Gaussian distribution. Since $\gamma$ is required to satisfy \eqref{gamma-def} in this case, we set $\gamma=0.99\cdot\gamma_{\max}$, where $\gamma_{\max}$ corresponds to the value on the right hand side of \eqref{gamma-def}. To observe the linear convergence behavior of ADMM \eqref{admm-N} for solving \eqref{toy-example}, we ran the ADMM \eqref{admm-N} for 100 iterations, and plotted the figures for $\|x_1-x_1^*\|$, $\|x_2-x_2^*\|$ and $\|x_3-x_3^*\|$, where the optimal solution $(x_1^*,x_2^*,x_3^*)$ was obtained via solving \eqref{toy-example} by the commercial software MOSEK \cite{Mosek}. From our numerical experiments, we observed the linear convergence behavior of ADMM \eqref{admm-N} for all the tested instances. In Figure \ref{fig:linear} we plot the convergence behavior for two specific instances which correspond to setting the seed for random function in Matlab to 0 and 1, respectively. From Figure \ref{fig:linear} we observed that $(x_1,x_2,x_3)$ globally converges to $(x_1^*,x_2^*,x_3^*)$ linearly. We also observed that $\gamma = 1.79\times 10^{-2}, \delta_1 = 1.01\times 10^{-2}$ and $\gamma =1.70\times 10^{-2}, \delta_1 = 1.01\times 10^{-2}$ in these two cases, and thus $\gamma$ and $\delta_1$ are not necessarily to be very small in practice.
\begin{figure}[htbp]
\centering{
\includegraphics[scale = 0.5]{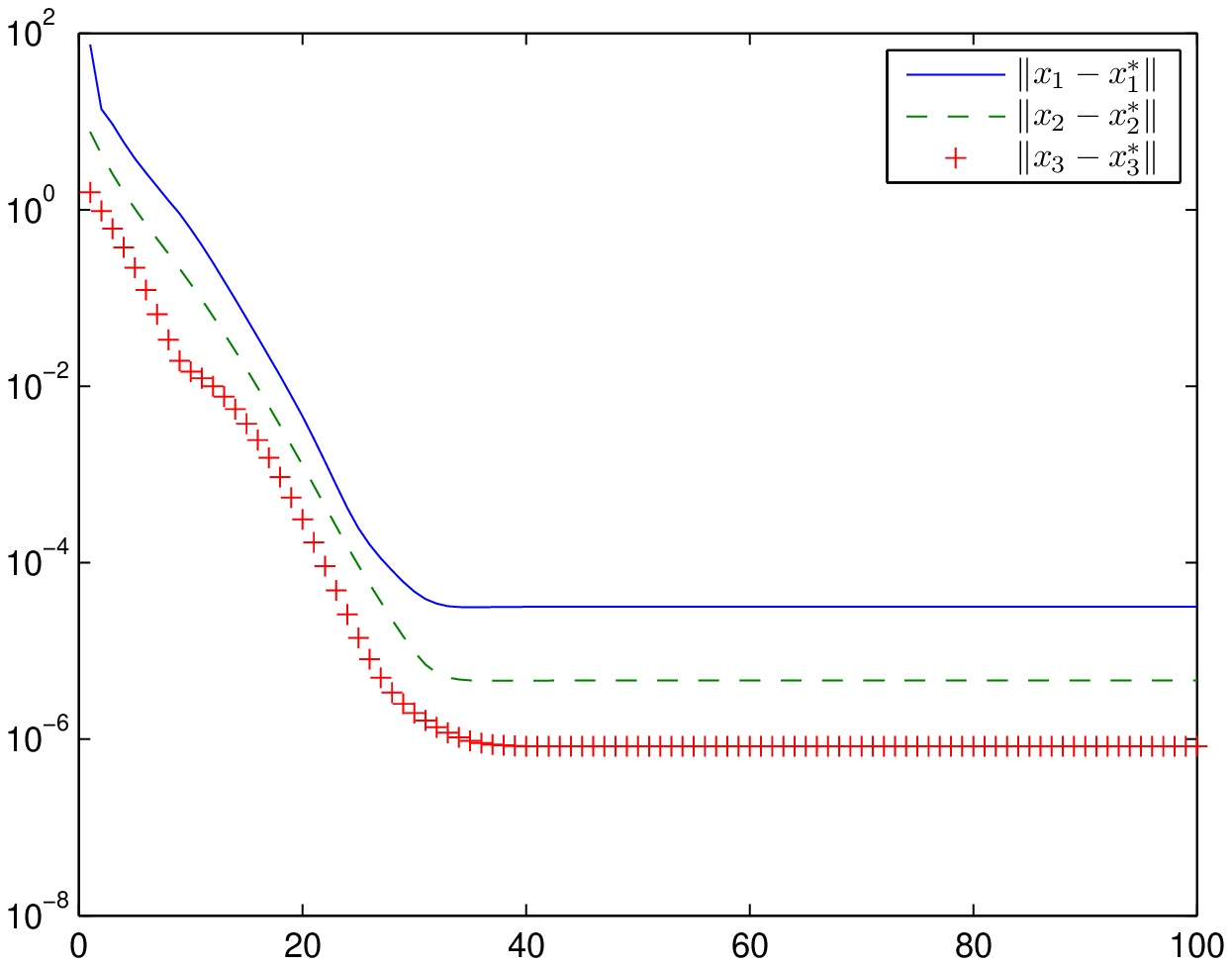}
\includegraphics[scale = 0.5]{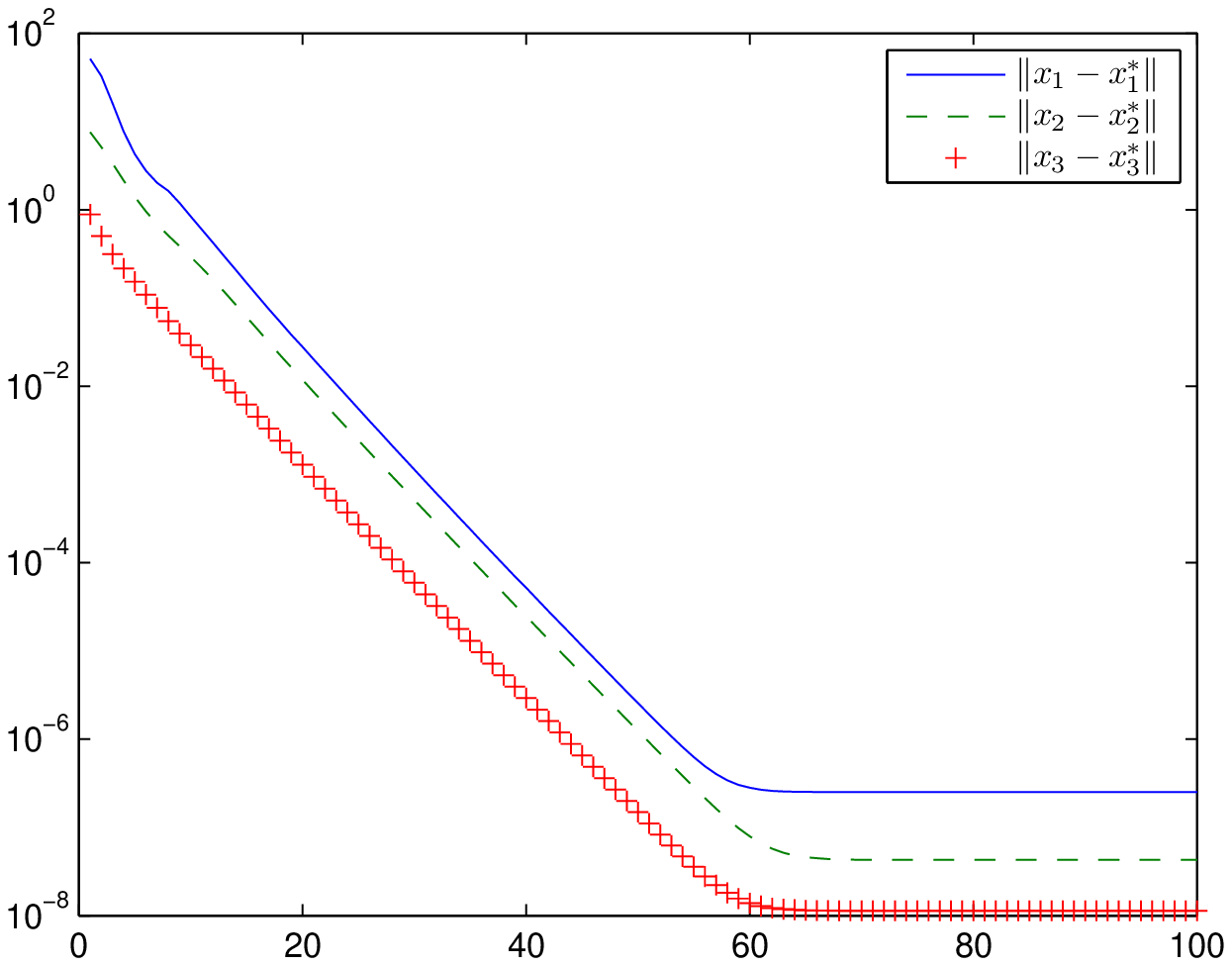}
}\caption{Linear convergence behavior of ADMM \eqref{admm-N} for solving \eqref{toy-example}. Left: seed for random function is set 0, $\gamma = 1.79\times 10^{-2}, \delta_1 = 1.01\times 10^{-2}$; Right: seed for random function is set 1, $\gamma=1.70\times 10^{-2}, \delta_1 = 1.01\times 10^{-2}$. }
\label{fig:linear}
\end{figure}

\section{Conclusions}\label{sec:conclusion}

In this paper we proved that the original ADMM for convex optimization with multi-block variables is linearly convergent under some conditions. In particular, we presented three scenarios under which a linear convergence rate holds for the ADMM; these conditions can be considered as extensions of the ones discussed in \cite{Deng-Yin-2012} for the 2-block ADMM. Convergence and complexity analysis for multi-block ADMM are important because the ADMM is widely used and acknowledged to be an efficient and effective practical solution method for large scale convex optimization models arising from image processing, statistics, machine learning, and so on.

\section*{Acknowledgements}

We would like to thank Mingyi Hong for useful discussions. We are grateful to the associate editor and two anonymous referees for constructive suggestions that improved the presentation of the paper.
Research of Shiqian Ma was supported in part by the Hong Kong Research Grants Council
General Research Fund Early Career Scheme (Project ID: CUHK 439513). Research of Shuzhong Zhang was supported in part by the National Science Foundation under Grant Number CMMI-1161242.

\bibliographystyle{plain}
\bibliography{admm3}

\end{document}